\documentclass[11pt]{article}
\usepackage{amssymb}
\usepackage{amsmath}
\usepackage{amsthm}

\numberwithin{equation}{section}

\newtheorem{theorem}{Theorem}[section]
\newtheorem{lemma}{Lemma}[section]

\newtheorem{remark}{Remark}[section]

\date{}

\title{\bf The Linear  Stability and Basic Reproduction Numbers for Autonomous FDEs}

\author{Xiao-Qiang Zhao\thanks{Research supported in part by the NSERC of Canada (RGPIN-2019-05648).}\\
	Department of Mathematics and Statistics\\
	Memorial University of Newfoundland\\
	St. John's, NL A1C  5S7, Canada\\
	E-mail:\,  zhao@mun.ca}

\begin{document}
\maketitle

\centerline{Dedicated to Professor Yihong Du on the occasion of his 60th birthday}

\begin{abstract}
In this paper, we first prove  the stability equivalence between a linear
autonomous and cooperative functional differential equation (FDE)
and its associated autonomous and cooperative system without time delay. Then we present the theory of 
basic reproduction number $\mathcal{R}_0$ for general autonomous 
FDEs. As an illustrative example, we
also establish the threshold dynamics for a time-delayed  population model of black-legged ticks in terms of $\mathcal{R}_0$.
\end{abstract}

\smallskip


\smallskip


\smallskip

\noindent {\bf Short title: }  Stability and $\mathcal{R}_0$   for 
autonomous FDEs

\def\rbb{\mathbb{R}}
\def\ml{\mathcal{L}}
 \def\th{\theta}
 \def\si{\sigma}  
 \def\mr{\mathcal{R}}

\section{Introduction}

The basic reproduction number (ratio) $\mathcal{R}_0$ is one of the most important concepts in population
biology. In epidemiology, $\mathcal{R}_0$ is the expected  number of secondary cases produced, in
a completely susceptible population, by a typical infective individual, and $\mathcal{R}_0$ is also
a commonly used measure of the effort needed to control an infectious disease.
The theory of $\mathcal{R}_0$ has been greatly developed for various evolution systems, see survey and review papers \cite{Hees, HSW,CD,van,Zhao} and references 
therein.  Among these research works,
Diekmann, Heesterbeek and Metz \cite{DHM} introduced  the next generation operators approach;
van den Driessche and Watmough \cite{DriesscheWatmough2002} gave a formula of $\mathcal{R}_0$ and proved the local 
stability in terms of $\mathcal{R}_0$ for compartmental 
ordinary differential equation (ODE) models; 
 Baca\"{e}r and and Guernaoui \cite{Bac1} proposed a general definition of $\mathcal{R}_0$ for population models in a periodic environment; Wang and Zhao \cite{WangZhao1} characterized $\mathcal{R}_0$ and 
 proved the local stability in terms of $\mathcal{R}_0$ for periodic compartmental ODE models;
Thieme \cite{ThiemeSIAP} presented the theory of spectral bounds and reproduction numbers for infinite-dimensional population structure; Wang and Zhao \cite{WangZhao2}  gave a biologically meaningful definition of $\mathcal{R}_0$ for reaction-diffusion models and found its relation to the principal eigenvalue of an associated elliptic eigenvalue problem; and
 Inaba \cite{Inaba} employed a generation evolution operator to give  a new definition of $\mathcal{R}_0$  for structured populations in heterogeneous environments.

Regarding  population models with time delay, 
Zhao \cite{ZhaoJDDE}  established the theory of $\mathcal{R}_0$ 
for periodic time-delayed compartmental systems 
where the internal transition  does not involve time
delay, and give a formula of $\mathcal{R}_0$ in the autonomous case
of  such systems (see \cite[Corollary 2.1]{ZhaoJDDE}). Liang, Zhang and Zhao \cite{LZZJDDE19}  further
generalized this theory to periodic abstract  FDEs 
so that it can be applied to spatial population models whose solution
maps are not compact. More recently,  Huang, Wu and Zhao \cite{HWZJDE}
extended the theory developed in \cite{LZZJDDE19} to more general periodic abstract FDEs where the internal transition term also has time delay.  Note that any autonomous evolution system can be regarded as
an $t_0$-periodic one for any given positive number $t_0$.
As a straightforward consequence of this theory, a general formula of $\mathcal{R}_0$
was provided for autonomous FDEs (see \cite[Corollary A.8]{HWZJDE}).
  
In this paper, we will first show that the stability of a linear
autonomous and cooperative FDE can be determined by that of an associated autonomous and cooperative system without time delay
from a perspective of the stability equivalence related to $\mathcal{R}_0$. This result further develops the earlier works in \cite{MS1991, Smith95,ThZh} and also of its own interest.
Then we will introduce next generation operators directly for general autonomous FDEs to define $\mathcal{R}_0$ rather than regarding them as time-periodic systems, and prove the stability  equivalence in terms of $\mathcal{R}_0$, which makes the formulas given in \cite[Corollary 2.1]{ZhaoJDDE} and 
\cite[Corollary A.8]{HWZJDE}  have more intuitive biological meaning.

The rest of this paper is organized as follows. In the next two sections, 
we consider autonomous FDEs on $\mathbb{R}^m$ and 
abstract autonomous FDEs on an ordered Banach space $X$, respectively.
In section 4, as an illustrative example, we briefly study a time-delayed  population model of black-legged ticks to obtain a threshold type result on its global dynamics in terms of  $\mathcal{R}_0$.

\section{Autonomous FDEs on $\mathbb{R}^m$}
Let $\tau\in\rbb_{+}$ be  given, and $C=C([-\tau, 0], \rbb^m)$ equipped with the maximum norm $\|\cdot\|_{C}$ and the positive cone $C_{+}=C([-\tau, 0], \rbb^m_{+})$. Then $(C, C_+)$ is an ordered Banach space. Let  $\ml(C, \rbb^m)$ be the space of all bounded and linear operators from $C$ to $\rbb^m$.
For any  $L\in  \ml(C, \rbb^m)$, we define $\hat{L}\in  \ml(\rbb^m, \rbb^m)$ by
$$
\hat{L}x=L(\hat{x}), \quad  \forall x\in \rbb^m,
$$
where $\hat{x}(\theta)=x, \forall \theta\in [-\tau,0]$.  Clearly, $\hat{L}$
can be regarded as an $m\times m$ matrix.  Recall that the stability modulus of an $m\times m$ matrix $A$ is defined as 
$$
s(A)=\max\big\{\Re(\lambda): \lambda \ \text{is an eigenvalue of  } A \big\}. 
$$

For a continuous function $u: [-\tau, \si)\to\rbb^m$ with $\si>0$, we define $u_t\in C$ by 
\[
u_t(\th)=u(t+\th), \quad \forall \, \th\in [-\tau, 0],
\]
for any $t\in [0, \si)$. 

Let $L\in  \ml(C, \rbb^m)$ be given. By the general theory of linear FDEs in \cite[Section 8.1]{Hale1993}, it follows that for any $\phi\in C$, the linear system 
\begin{equation}
\dfrac{du(t)}{dt}=L(u_t)
\label{autoE}
\end{equation}
admits a unique solution $u(t,\phi)$ on $[0, \infty)$ with $u_0=\phi$. 
Let $T(t)$ be the solution semigroup of linear system \eqref{autoE}, that is,
\begin{equation}
T(t)\phi=u_t(\phi), \quad \forall\phi\in C, \, t\ge 0.
\end{equation}
We define  the stability modulus of linear system \eqref{autoE} as
$$
s(L)=\max\big\{\Re(\lambda): \lambda \ \text{is an eigenvalue of system \eqref{autoE}}\big\}. 
$$

Assume that $L=(L_1,\cdots, L_m)$ satisfies the following quasimonotone  condition:
\begin{enumerate}
	\item[(K)]  $L_i(\phi)\geq 0$ whenever $\phi\geq 0$ and $\phi_i(0)=0$.
	\end{enumerate}
 Then \cite[Theorem 5.1.1]{Smith95} implies that $T(t)$ is monotone semiflow on 
$C$. Further, if $L$ satisfies (K) and the irreducibility conditions (R) and (I), then $s(L)$ is the principal eigenvalue of system \eqref{autoE} (see \cite[Theorem 5.5.1]{Smith95}).

The following result improves \cite[Corollary 5.5.2]{Smith95}, where
the irreducibility conditions (R) and (I) were assumed.
\begin{theorem}\label{stability}
Assume that $L=(L_1,\cdots, L_m)$ satisfies 
condition (K). Then s(L) and $s(\hat{L})$ have the same sign.
\end{theorem}

\begin{proof}
Given $t_0>0$, we may regard system \eqref{autoE} as a linear $t_0$-periodic system.  By the spectral mapping theorem for linear autonomous FDEs (see, e.g.,
	\cite[Lemma 7.6.1]{Hale1993}),  it follows that 
	\begin{equation*}
	r(T(t))=e^{s(L)t},\quad \forall t>0,
	\end{equation*}
	and hence,  $sign(s(L))=sign(r(T(t_0))-1)$.

Since 	$L$ satisfies (K), \cite[Lemma 5.1.2]{Smith95} implies that 
there exist a diagonal matrix $D=diag(a_1,\cdots,a_m)$ and $\bar{L}
\in \ml(C, \rbb^m)$ with $\bar{L}(C_+)\subset \mathbb{R}^m_+$ such that
$$
L(\phi)=D\phi(0)+\bar{L}(\phi), \quad \forall \phi\in C.
$$
Fix a real number  $a>\max\{0, a_1, \cdots,a_m\}$. Thus, we can write 
$L$ as 
$$
L=(D-aI)\phi(0)+(a\phi(0)+\bar{L}(\phi)), \quad \forall \phi\in C.
$$
By  \cite[Theorem 2.1 and Corollary 2.1]{ZhaoJDDE} with $\mathcal{F}(\phi)=a\phi(0)+\bar{L}(\phi)$
and $\mathcal{V}=aI-D$, it follows that 
\begin{equation}\label{PerFDER0}
sign(s(L))=sign(r(T(t_0))-1)=sign(r(\hat{\mathcal{F}}\mathcal{V}^{-1})-1).
\end{equation}
On the other hand, the well-known result on basic reproduction numbers
for ODE systems \cite[Theorem 2]{DriesscheWatmough2002} (see also \cite[Theorem 2.3]{ThiemeSIAP}) gives rise to 
\begin{equation}\label{ODER0}
sign(r(\hat{\mathcal{F}}\mathcal{V}^{-1})-1)
=sign (s(\hat{\mathcal{F}}-\mathcal{V}))=sign (s(\hat{L})).
\end{equation}
Now \eqref{PerFDER0} and \eqref{ODER0} show that
$sign(s(L))=sign (s(\hat{L}))$.
\end{proof}

Next, we consider the following autonomous linear
 FDE:
\begin{equation}
\dfrac{du(t)}{dt}=F(u_t)-V(u_t),
\label{eqn1.1}
\end{equation}
where $F\in \ml(C, \rbb^m)$ and 
$ V\in \ml(C, \rbb^m)$.
System \eqref{eqn1.1} may come from the equations of infectious variables in  the
linearization of a given time-delayed compartmental epidemic model at a
disease-free equilibrium.
As such, $m$ is the total number of the infectious compartments, and the newly infected individuals at
time $t$ depend linearly on the infectious individuals over the time interval
$[t-\tau, t]$, which is described by $F(u_t)$. Further, the internal transition of individuals in the infectious compartments (e.g., natural and disease-induced deaths, and movements among compartments) is governed by the linear time-delayed  system:
\begin{equation}
\label{internal_sys}
\dfrac{du(t)}{dt}=-V(u_t).
\end{equation}

Without loss of generality, we assume that $F: C\to \mathbb{R}^m$ is given by
$$
F(\phi)=\int_{-\tau}^0d[\eta(\theta)]\phi(\theta),\quad  \forall \phi\in C,
$$
where $\eta (\theta)$ is an $m\times m$ matrix function which is measurable in $\theta\in \mathbb{R}$ and normalized so that $\eta(\theta)=0$ for all $\theta\geq 0$
and $\eta(\theta)=\eta(-\tau)$ for all $\theta\leq  -\tau$. Moreover, $\eta(\theta)$  is
continuous from the left in $\theta$ on $(-\tau,0)$, and the variation of
$\eta(\cdot)$ on $[-\tau,0]$ is bounded.
Clearly, $F$ is also well defined on $L^1([-\tau, 0],\rbb^m)$.

Let $Q(t)$ and $\Phi(t)$ be the solution semigroups of linear systems \eqref{eqn1.1} and \eqref{internal_sys}, respectively.
 According to \cite[Theorem 1]{SchJDE}, there is a family of continuous matrix-valued functions $X(t)$ on $[0,\infty)$
such that for any $\phi\in C$ and $g\in C([0,\infty), \mathbb{R}^m)$, the unique solution $y(t)$ with initial data $y_0=\phi$ of the following inhomogeneous system 
\begin{equation}
\dfrac{du(t)}{dt}=-V(u_t) +g(t), \quad t\geq 0
\end{equation}
satisfies 
$$
y(t)=[\Phi(t)\phi](0)+\int_0^tX(t-s)g(s)ds,  \quad \forall t\geq 0.
$$
Choose a family of functions  $h^\epsilon\in C^\infty([-\tau,0], [0,1])$
with $\epsilon\in(0,\tau]$ such that
$h^\epsilon$ is nondecreasing on $[-\tau,0]$;
$h^\epsilon(\theta)=0,\,  \forall \theta\in [-\tau, -\epsilon]$ and $h^\epsilon(0)=1$; and 
$h^{\tilde{\epsilon}}(\theta)\leq h^{\epsilon}(\theta),\ \forall 0<\tilde{\epsilon}\leq\epsilon\leq \tau$ and $\theta\in[-\tau,0]$.  
Define a family of linear operators on $\mathbb{R}^m$
 by 
$$
X^\epsilon(t)x=[\Phi(t)h^\epsilon(\cdot)x](0), \, \forall t\geq 0,
x\in \mathbb{R}^m.
$$
It then follows from \cite[Theorem 1]{SchJDE} that 
\begin{equation}
X(t)x=\lim_{\epsilon\rightarrow 0^+}X^\epsilon(t)x,\,  \, \forall  t\geq 0, x\in \mathbb{R}^m.
\end{equation}
Thus, $z(t):=X(t)x$ may be regarded as a solution of system \eqref{internal_sys} on $[0, \infty)$ with the initial data satisfying $z_0(\theta)=0$ for all $\theta\in [-\tau, 0)$ and $z_0(0)=x$.

Let $s(-V)$ be the stability modulus of linear system 
\eqref{internal_sys}. Then we have the following observation.

\begin{lemma} 
	 If $s(-V)<0$, then $\int_0^{\infty}X(s)ds=\hat{V}^{-1}$.
\end{lemma}

\begin{proof}
	For any given $x\in \rbb^m$, it is easy to see that the autonomous system
	\begin{equation}\label{inverse}
	\dfrac{du(t)}{dt}=-V(u_t) +x
	\end{equation}
	has a unique equilibrium $u^*=\hat{V}^{-1}x$.  
	Since $s(-V)<0$,  we conclude that $u^*$ is globally attractive
	for system \eqref{inverse} in $C$. Note that
	$z(t):=\int_0^tX(t-s)xds=\left(\int_0^tX(s)ds\right)x$ is the unique solution on $[0,\infty)$ of \eqref{inverse} with $z_0=0$.  It then follows that 
	$$
	\lim_{t\to \infty}z(t)=\left(\int_0^{\infty}X(s)ds\right) x=u^*=\hat{V}^{-1}x.
	$$
	Since $x\in \rbb^m$ is arbitrary, we obtain $\int_0^{\infty}X(s)ds=\hat{V}^{-1}$.
\end{proof}

To introduce the basic reproduction number for system \eqref{eqn1.1}, throughout this section we assume that 
\begin{enumerate}
	\item[(A1)]  $F$ is positive in the sense that 
	$F(C_+)\subset \mathbb{R}^m_+$.
	\item[(A2)] $-V$ satisfies  condition (K) and $s(-V)<0$.
\end{enumerate}

Let $x\in\rbb^m_{+}$ be the initial infected individuals distributed among  $m$ compartments. The distribution of these individuals under the internal evolution at time $t\ge 0$ is $X(t)x$. Define a function 
\begin{equation*}
\tilde{u}(t)=
\begin{cases}
X(t)x, & \forall \, t\ge 0,\\
0, &  \forall \, -\tau\le t< 0.
\end{cases}
\end{equation*}
Then the new infection of the infected individuals at time $t$ is $F(\tilde{u}_t)$.  Thus, the distribution of the total new infection is 
\begin{equation*}
\begin{split}
\int_{0}^{\infty}F(\tilde{u}_t)dt &= \int_0^{\infty}\left(\int_{-\tau}^0d\eta(\th)\tilde{u}(t+\th)\right)dt\\
&=\int_{-\tau}^{0}d\eta(\theta)\int_0^{\infty}\tilde{u}(t+\theta)dt\\
&=\int_{-\tau}^{0}d\eta(\theta)\int_{\theta}^{\infty}\tilde{u}(s)ds\\
&=\int_{-\tau}^{0}d\eta(\theta)\left(\int_{\theta}^{0}\tilde{u}(s)ds+\int_{0}^{\infty}\tilde{u}(s)ds\right)\\
&=\int_{-\tau}^{0}d\eta(\theta)\int_{0}^{\infty}X(s)xds\\
&=\hat{F}(\hat{V}^{-1}x).
\end{split}
\end{equation*}
  It follows that 
the next generation matrix for system \eqref{eqn1.1} is 
$\hat{F}\hat{V}^{-1}$. Accordingly, we define  the basic reproduction number for system \eqref{eqn1.1} as the spectral radius of 
$\hat{F}\hat{V}^{-1}$, that is,
\begin{equation}\label{ratio}
\mr_0=r\big(\hat{F}\hat{V}^{-1}\big).
\end{equation}

	Let $\lambda^{*}=s(F-V)$ be  the stability modulus of linear system \eqref{eqn1.1}.  Then we have the following stability result
	in terms of $\mr_0$.
	
\begin{theorem}
$\mr_0-1$ and $\lambda^{*}$ have the same sign.
\label{thm-1}
\end{theorem}

\begin{proof}
Since $\mr_0=r\big(\hat{F}\hat{V}^{-1}\big)$, it follows from \cite[Theorem 2]{DriesscheWatmough2002} (see also \cite[Theorem 2.3]{ThiemeSIAP}) that 
 $$sign(\mr_0-1)=sign (s(\hat{F}-\hat{V})).$$
By Theorem \ref{stability} with $L=F-V$, we see that 
 $$
 sign(s(F-V))=sign(s(\widehat{F-V}))=sign(s(\hat{F}-\hat{V})).
$$
Thus,  we have $sign(\mr_0-1)=sign(s(F-V))=\lambda^{*}$.
\end{proof}

Let $Q_{\mu}(t)$ be the solution semigroup of the following linear system 
with parameter $\mu>0$:
\begin{equation}
\frac{du}{dt}=\frac{1}{\mu}F(u_t)-V(u_t)
\label{eqn2}
\end{equation}	
and $\lambda^*(\mu)=s(\frac{1}{\mu}F-V)$ be its stability modulus. Clearly,
$Q_{1}(t)=Q(t)$ and $\lambda^*(1)=\lambda^*$.
As a consequence of Theorem \ref{thm-1}, we have the following observation.
\begin{theorem}
	 If $\mr_0>0$, then $\mu=\mr_0$ is the unique positive solution of 
	$r(Q_{\mu}(t_0))=1$ for any given $t_0>0$, and also 
	 the unique positive solution of $\lambda^*(\mu)=0$.
\end{theorem}

\begin{proof}
Let $\mr_0 (\mu)$ be the basic reproduction number of system \eqref{eqn2}.
Clearly, there holds
$$
\mr_0(\mu)=\frac{1}{\mu}\mr_0, \, \, \forall \mu>0.
$$
Since $\mr_0>0$,  $\mr_0$ is the unique solution of $\mr_0(\mu)=1$.
Let $t_0>0$ be given. 
By Theorem \ref{thm-1} and the spectral mapping theorem for linear autonomous FDEs (see \cite[Lemma 7.6.1]{Hale1993}), it then follows
that 
$$
sign(\mr_0(\mu)-1)=sign(\lambda^*(\mu))=sign(r(Q_{\mu}(t_0))-1).
$$
This implies two desired statements.
\end{proof}
	
\begin{remark}
In view of \eqref{ratio} and the theory of $\mr_0$ for ODE systems in \cite{DriesscheWatmough2002}, it follows that 
	$\mr_0$ defined for FDE system \eqref{eqn1.1} 
	is  also the basic reproduction number of the following ODE system:
	\[
	\frac{du}{dt}=\hat{F}u-\hat{V}u=\big(\hat{F}-\hat{V}\big)u.
	\]
\end{remark}	

\begin{remark}
In the case where $\tau=0$, $\mr_0$ in \eqref{ratio} reduces to the formula of $\mr_0$ for ODE systems in \cite{DriesscheWatmough2002}, and in the case where 
$V(\phi)=V\phi(0)$ for a square matrix $V$, 
$\mr_0$ in \eqref{ratio} is also consistent with the formula of $\mr_0$ 
given in \cite[Corollary 2.1]{ZhaoJDDE}.
\end{remark}	

\section{Abstract autonomous FDEs}

In this section, we extend the results in section 2 to a large class of abstract  autonomous FDEs. We start with a simple  observation on the spectral radius of positive linear operators.

\begin{lemma} \label{radius}
Let $(W,P)$ be an ordered Banach space with  the positive cone $P$
being solid (i.e., $Int(P)\ne \emptyset$), and $M$ be a bounded and linear operator on $W$. If $M$ is positive (i.e., $M(P)\subset P$) and $Me\gg 0$ for some $e\gg 0$, then $r(M)>0$.
\end{lemma}

\begin{proof}
Since 	$\lim_{\delta \to 0}\delta e=0\ll Me$,  we can fix a sufficiently small number $\delta >0$ such that $Me\gg \delta e$. Now we show that $r(M)\geq \delta$. Assume, by contradiction, that $r(M)<\delta$. Since $$
\lim_{n\to \infty}
\|(\frac 1{\delta}M)^n\|^{\frac 1n}=r(\frac 1{\delta}M)=\frac 1{\delta}r(M)<1,$$
  there exists $n_0>0$ such that 
$\|(\frac 1{\delta}M)^{n}\|<1$ for all $n\geq n_0$.  Let $\bar{M}=
(\frac 1{\delta}M)^{n_0}$. Clearly,  $\|\bar{M}\|<1$ and
$h:=\bar{M}e- e\gg 0$. Since $(I-\bar{M})e=-h$,
it follows that 
$$
e=(I-\bar{M})^{-1}(-h)=-\sum_{i=0}^{\infty}\bar{M}^i(h)\leq 0,
$$
which contradicts $e\gg 0$.
\end{proof}

Let $(X, X_+)$ be an ordered Banach space with the positive cone $X_+$
being normal and solid. Let $\tau\in\rbb_{+}$ be  given, and $E=C([-\tau, 0], X)$ equipped with the maximum norm $\|\cdot\|_{E}$ and the positive cone $E_{+}=C([-\tau, 0], X_+)$. Then $(E, E_+)$ is an ordered Banach space. Let  $\ml(E, X)$ be the space of all bounded and linear operators from $E$ to $X$.
For any  $L\in  \ml(E, X)$, we define $\hat{L}\in  \ml(X, X)$ by
$$
\hat{L}x=L(\hat{x}), \quad  \forall x\in X,
$$
where $\hat{x}(\theta)=x, \forall \theta\in [-\tau,0]$.

We consider the following abstract autonomous FDE:
\begin{equation}
\dfrac{du(t)}{dt}=Au(t)+B(u_t)
\label{Abs1}
\end{equation}
where $A$ is a closed linear operator in $X$ with a dense domain 
$D(A)$ and $B\in  \ml(E, X)$.  Assume that 
\begin{enumerate}
	\item[(H)] $A: D(A)\to X$ generates a strongly continuous positive semigroup $T_A(t)$ on $X$, 
	and $B$ is positive in the sense that $B(E_+)\subset X_+$.
\end{enumerate}

By the standard semigroup theory (see, e.g., \cite{MS1990}), it follows that 
for any $\phi\in E$, system \eqref{Abs1} has a unique mild
solution $u(t,\phi)$ on $[0,\infty)$ with $u_0=\phi$, and its solution
maps generate a positive semigroup $\mathcal{T}(t)$ on $E$. 

Recall that  the exponential growth bound of the semigroup $\mathcal{T}(t)$ is defined as
$$\omega(\mathcal{T})=\inf\{\tilde{\omega}: \exists M\geq 1\ \text{such that $\|\mathcal{T}(t)\|\leq Me^{\tilde{\omega}t}$, $\forall  t\geq 0$}\},
$$
and  the spectral bound of  the closed linear operator $A+\hat{B}$ in $X$
is defined as 
$$
s(A+\hat{B})=\sup\{Re\, \lambda:   \, \, \lambda\in \sigma (A+\hat{B})\},
$$
where $\sigma (A+\hat{B})$ is  the spectrum of  $A+\hat{B}$.

The following result is a generalization of Theorem \ref{stability} to  
abstract FDE \eqref {Abs1} on $X$.

\begin{theorem}\label{stabilityAbs}
	Assume that $A$ and $B$ satisfy (H), and $T_A(t)$ is compact on 
	$X$ for each $t>0$. Then $\omega(\mathcal{T})$ and 
	$s(A+\hat{B})$  have the same sign.
\end{theorem}

\begin{proof}
We first choose  a large real number $b>0$ such that the exponential 
growth bound of the semigroup $e^{-bt}T_A(t)$ is negative. Let
$-V:=A-bI$ and $F(\phi):=b\phi (0)+B(\phi), \forall \phi\in E$. 
Then we
can write equation \eqref{Abs1} as 
\begin{equation}
\dfrac{du(t)}{dt}=F(u_t)-Vu(t).
\label{Abs2}
\end{equation}
For any given $t_0>0$, we regard  \eqref{Abs1}  as an
$t_0$-periodic equation.  In view of \cite[Proposition A.2]{ThiemeSIAP},
we have 
$$
\omega(\mathcal{T})=\frac{\ln r(\mathcal{T}(t_0))}{t_0}.
$$
We fix an element $e\in Int(X_+)$. Since $\lim_{t\to 0^+}T_A(t)e=
e\gg 0$, there exists $\delta >0$ such that $T_A(t)e\gg 0$ for 
all $t\in [0,\delta]$. From \cite[Theorem 3.12]{ThiemeSIAP}, we further see that 
$$
(bI-A)^{-1}e=-(A-bI)^{-1}e=\int_0^{\infty}e^{-bt}T_A(t)edt\gg 0.
$$
This, together with Lemma \ref{radius}, implies that 
$r((bI+\hat{B})\circ (bI-A)^{-1})>0$.
Let $B_{t_0}$ be the ordered Banach space of all continuous and 
$t_0$-periodic functions from $\mathbb{R}$ to $X$. Following
\cite{LZZJDDE19}, we define a linear  operator $\tilde{\mathfrak{L}}$ on $B_{t_0}$ by 
$$
[\tilde{\mathfrak{L}}(v)](t)=F\int_0^{\infty}e^{-bs}T_A(s)v(t-s+\cdot)ds,
\quad \forall t\in \mathbb{R}, \,  v\in B_{t_0}.
$$
It is easy to see that 
$$
\tilde{\mathfrak{L}}(x)=F\int_0^{\infty}e^{-bs}T_A(s)xds=
(bI+\hat{B})\circ (bI-A)^{-1}x, \, \,
\forall x\in X.
$$
In view of  \cite[Lemma 2.4]{LZZJDDE19}, we obtain
$r(\tilde{\mathfrak{L}})=r((bI+\hat{B})\circ (bI-A)^{-1})$.
By \cite[Theorem 3.7 and Proposition 3.9]{LZZJDDE19}, it then follows that 
\begin{equation}\label{AbsHWZ}
sign(\omega(\mathcal{T}))=sign(r(\mathcal{T}(t_0))-1)=
sign(r((bI+\hat{B})\circ (bI-A)^{-1})-1).
\end{equation}
Let $\mathcal{A}=A+\hat{B}$, $\mathcal{B}=A-bI$ and $\mathcal{C}=bI+\hat{B}$. Clearly, $\mathcal{A}=\mathcal{C}+\mathcal{B}$ and 
$s(\mathcal{B})<0$.  Since $\mathcal{A}$ and $\mathcal{B}$ generates positive  $C_0$-semigroups on $X$, \cite[Theorem 3.12]{ThiemeSIAP} implies that they are resolvent-positive. Thus,  \cite[Theorem 3.5]{ThiemeSIAP} gives rise to
\begin{equation}\label{AbsThieme}
sign(s(A+\hat{B}))=sign (s(\mathcal{A}))=sign (r(-\mathcal{C}\mathcal{B}^{-1})-1).
\end{equation}
Since $-\mathcal{C}\mathcal{B}^{-1}=(bI+\hat{B})\circ (bI-A)^{-1}$, it 
follows from \eqref{AbsHWZ} and \eqref{AbsThieme} that $\omega(\mathcal{T})$ and $s(A+\hat{B})$  have the same sign.
\end{proof}

Let $T_N(t)$ be the strongly continuous semigroup generated by 
a closed linear operator $N$ in $X$ with a dense domain $D(N)$, and
$L, \mathcal{F}\in  \ml(E, X)$. Next, we consider the following autonomous linear  FDE:
\begin{equation}\label{Abs3}
\dfrac{du}{dt}=Nu(t)+L(u_t)+\mathcal{F}(u_t)=
\mathcal{F}(u_t)-\mathcal{V}(u_t)
\end{equation}
where  $-\mathcal{V}\phi:=N\phi(0)+L\phi$.  It then follows that 
\eqref{Abs3} generates a semigroup $U(t)$ on $E$, and the autonomous
linear FDE
\begin{equation}\label{Abs4}
\dfrac{du}{dt}=-\mathcal{V}(u_t)
\end{equation}
generates a semigroup $\Psi(t)$ on $E$, respectively. 

To introduce the basic reproduction number for system \eqref{Abs3}, throughout this section we assume that 
\begin{enumerate}
	\item[(H1)] $\mathcal{F}$ is positive in the sense that 
	$\mathcal{F}(E_+)\subset X_+$.
	\item[(H2)] $T_N(t)$ is a positive semigroup, $L$ is positive,  and $\omega(\Psi)<0$.
	\end{enumerate}

With the help of \cite[Theorem 1]{SchJDE}, as applied to \eqref{Abs4}, we may use the essentially same arguments
as those for \eqref {eqn1.1} to derive the next generation operator as $\hat{\mathcal{F}}\circ\hat{\mathcal{V}}^{-1}$, where 
 $-\hat{\mathcal{V}}:=N+\hat{L}$. Thus, we define the basic reproduction 
 number for system  \eqref{Abs3} to be 
$\mathcal{R}_0=r(\hat{\mathcal{F}}\circ\hat{\mathcal{V}}^{-1})$.

The following result shows that the stability of the zero solution for system  \eqref{Abs3} can be 
determined by the sign of $\mr_0-1$.

\begin{theorem}	\label{AbsTh1R0}
Assume that $T_N(t)$ is compact on 
$X$ for each $t>0$.  Then $\mr_0-1$ and $\omega(U)$ have the same sign.
\end{theorem}

\begin{proof}
	By \cite[Theorem 3.5]{ThiemeSIAP} with 
	$\mathcal{B}=-\hat{\mathcal{V}}$ and $\mathcal{C}=\hat{\mathcal{F}}$, 
 it follows from  that 
	$$
sign(\mr_0-1)=sign(r(\hat{\mathcal{F}}\circ\hat{\mathcal{V}}^{-1})-1)
	=sign (s(\hat{\mathcal{F}}-\hat{\mathcal{V}})).
	$$
On the other hand, Theorem \ref{stabilityAbs} with 
$A=N$ and $B=L+\mathcal{F}$ implies that 
	$$
	sign(\omega(U))=sign(s(N+\hat{L}+\hat{\mathcal{F}}))=sign (s(\hat{\mathcal{F}}-\hat{\mathcal{V}})).
	$$
Thus,  we have $sign(\mr_0-1)=sign(\omega(U))$.
\end{proof}

\begin{remark}
Since $\omega(U)=\frac{\ln r(U(t_0))}{t_0}$ for any given 
$t_0>0$,  it is easy to see that  Theorem \ref{AbsTh1R0} is a
straightforward consequence of \cite[Theorem A.5, Corollary A.8 and Proposition A.9]{HWZJDE}.
\end{remark}

Let $U_{\mu}(t)$ be the solution semigroup of the following linear system 
with parameter $\mu>0$:
\begin{equation}
\frac{du}{dt}=\frac{1}{\mu}\mathcal{F}(u_t)-\mathcal{V}(u_t).
\label{Abseqn2}
\end{equation}	
Then we have the following result which characterizes $\mr_0$.

\begin{theorem}\label{AbsCompute}
	Assume that $T_N(t)$ is compact on 
	$X$ for each $t>0$. If $\mr_0>0$, then $\mu=\mr_0$ is the unique positive solution of 
	$r(U_{\mu}(t_0))=1$ for any given $t_0>0$, and also 
	the unique positive solution of $\omega(U_{\mu}) =0$.
\end{theorem}

\begin{proof}
	Let $\mr_0 (\mu)$ be the basic reproduction number of system \eqref{Abseqn2}. It then follows that 
	$$
	\mr_0(\mu)=\frac{1}{\mu}\mr_0, \, \, \forall \mu>0.
	$$
	Since $\mr_0>0$,  $\mr_0$ is the unique solution of $\mr_0(\mu)=1$.
	For any given  $t_0>0$, there holds  $\omega(U_{\mu})=\frac{\ln r(U_{\mu}(t_0))}{t_0}$.
	In view of Theorem \ref{AbsTh1R0}, we have 
	$$
	sign(\mr_0(\mu)-1)=sign(\omega(U_{\mu}))=sign(r(U_{\mu}(t_0))-1).
	$$
	This gives rise to the desired two statements.
\end{proof}

\begin{remark}
	For any given $\mu>0$, we can use \cite[Lemma 2.5]{LZZJDDE19}
	to compute $r(U_{\mu}(t_0))$ numerically. Thus, the method of bisection can be employed to solve 
	$r(U_{\mu}(t_0))=1$, which gives the value of $\mr_0$
	due to Theorem \ref{AbsCompute}.
\end{remark}	

\begin{remark}
	In Theorems \ref {stabilityAbs}, \ref{AbsTh1R0} and \ref{AbsCompute},
	the compactness condition for the semigroups $T_A(t)$ and $T_N(t)$
	can be weakened via the verification of assumptions (H3), (H4) and (H5)
	in \cite{LZZJDDE19} (see also (C3), (C4) and (C5) in \cite{HWZJDE}).
\end{remark}	

\section{An application}

In this section, we apply the theory of $\mathcal{R}_0$ in section 2 to a time-delayed population model and obtain a threshold type result on its global dynamics in terms of 
$\mathcal{R}_0$.  Clearly,  one can also apply the theory of 
$\mathcal{R}_0$ in section 3 to some reaction-diffusion 
models with time delay. 

Gourley et al.  \cite{Gourley2018} proposed a nonlocal spatial model of black-legged ticks to study the role of white-tailed deer in their
geographic spread. The spatially homogeneous version of this model 
is governed by the following time-delayed differential system:
\begin{align}
& L'(t)=br_4e^{-d_4\tau_1}A_f(t-\tau_1)-(d_1+r_1)L(t), \nonumber\\
&N'(t)=r_1g(L(t))-(d_2+r_2)N(t),\nonumber\\
&A_q'(t)=r_2N(t)-(d_3+r_3)A_q(t),\label{ApplE1}\\
&A_f'(t)=\frac{r_3}{2}e^{-d_3\tau_2}A_q(t-\tau_2)-(d_4+r_4)A_f(t), \nonumber
\end{align}
where  $L(t)$, $N(t)$, $A_q(t)$ and $A_f(t)$ are the population densities of larvae, nymphs, questing adults and female fed adults at time $t$, respectively.  All parameters $b$, $r_i$, $d_i$ and $\tau_i$  are positive 
numbers, and we refer to Table 1 in \cite{Gourley2018} for their biological 
meanings. The nonlinear function $g(L)=\frac{N_{cap}L}{h+L}$ with two positive constants $N_{cap}$ and $h$.

Linearizing system \eqref{ApplE1} at  $(0,0,0,0)$, we obtain the following 
linear system:
\begin{align}
& L'(t)=br_4e^{-d_4\tau_1}A_f(t-\tau_1)-(d_1+r_1)L(t), \nonumber\\
&N'(t)=r_1g'(0)L(t)-(d_2+r_2)N(t),\nonumber\\
&A_q'(t)=r_2N(t)-(d_3+r_3)A_q(t),\label{ApplE2}\\
&A_f'(t)=\frac{r_3}{2}e^{-d_3\tau_2}A_q(t-\tau_2)-(d_4+r_4)A_f(t), \nonumber
\end{align}
Let $\tau=\max\{\tau_1, \tau_2\}$  and $C=C([-\tau,0],\mathbb{R}^4)$.
We define $F, V\in \mathcal{L}(C, \mathbb{R}^4)$ as follows 
$$
F(\phi)=\left(br_4e^{-d_4\tau_1}\phi_4(-\tau_1),0, 0, 0\right)^T,
\quad  \forall \phi=(\phi_1, \phi_2, \phi_3, \phi_4)\in C,
$$ 
and
$$
V(\phi)=\left(
\begin{array}{c}
(d_1+r_1)\phi_1(0)\\
-r_1g'(0)\phi_1(0)+(d_2+r_2)\phi_2(0)\\
-r_2\phi_2(0)+(d_3+r_3)\phi_3(0)\\
-\frac{r_3}{2}e^{-d_3\tau_2}\phi_3(-\tau_2)+(d_4+r_4)\phi_4(0)
\end{array}
\right), \quad  \forall \phi=(\phi_1, \phi_2, \phi_3, \phi_4)\in C.
$$
Then system \eqref{ApplE2} is of the form \eqref{eqn1.1}. It is easy to see 
that  
$$
\hat{F}=\left(
\begin{array}{cccc}
0 & 0 &0 &br_4e^{-d_4\tau_1}  \\
0& 0 & 0  &0\\
0& 0 &0   &0\\
0& 0 & 0  &0
\end{array}
\right)
$$
and
$$
\hat{V}=\left(
\begin{array}{cccc}
d_1+r_1 & 0 &0   &0 \\
-r_1g'(0)& d_2+r_2 & 0  &0\\
0& -r_2 &d_3+r_3  &0\\
0& 0 &-\frac{r_3}{2}e^{-d_3\tau_2}  &d_4+r_4
\end{array}
\right).
$$
By \eqref{ratio} and a straightforward computation, we obtain
\begin{equation}\label{ApplR0}
\mr_0=r\big(\hat{F}\hat{V}^{-1}\big)=\frac 12 bg'(0)e^{-(d_4\tau_1+d_3\tau_2)}\prod_{i=1}^4\frac{r_i}{d_i+r_i}.
\end{equation}

Note that  the function $g(x)$ is strictly sublinear (i.e, subhomogeneous) on $[0, \infty)$ in the sense that $g(sx)>sg(x)$ for all $x>0$ and $s\in (0,1)$. It follows that the solution maps of  system \eqref{ApplE1} 
are  sublinear (subhomogeneous) on $C([-\tau,0],\mathbb{R}^4_+)$  (see \cite{ZhaoJing}).  Further, one can  easily  verify that system \eqref{ApplE1} has a unique positive equilibrium $u^*$ in the case where $\mathcal{R}_0>1$.

 As a consequence of Theorems 
\ref{stability} and \ref{thm-1} and \cite[Theorem 3.2]{ZhaoJing}, we  have the following threshold dynamics for system \eqref{ApplE1}.

\begin{theorem}\label{Exampe}
Let $\mathcal{R}_0$ be given in \eqref{ApplR0}. Then the following 
statements are valid.
\begin{enumerate}
\item[(i)] If  $\mathcal{R}_0\leq 1$, then the zero solution is globally
asymptotically stable in $C([-\tau,0],\mathbb{R}^4_+)$;

\item[(ii)] If $\mathcal{R}_0>1$, then system \eqref{ApplE1} admits
a unique positive equilibrium $u^*$ and $u^*$ is globally asymptotically 
stable in $C([-\tau,0],\mathbb{R}^4_+)\setminus \{0\}$.
\end{enumerate}
\end{theorem}


\begin{thebibliography}{00}
	
\bibitem{Bac1} N. Baca\"{e}r  and  S. Guernaoui,  The epidemic threshold of vector-borne diseases with seasonality, {\it J. Math. Biol.}, 53(2006), 421-436.
	
\bibitem{CD} J. M.  Cushing  and O. Diekmann, The many guises of  $R_0$  (a didactic note),	{\it J. Theoret. Biol.}, 404 (2016), 295-302.
	
\bibitem{DHM} O. Diekmann, J. A. P. Heesterbeek and  J. A. J. Metz,  On the definition and the computation of the basic reproduction ratio
	$R_{0}$ in the models for infectious disease in heterogeneous
	populations, {\it J. Math. Biol.}, 28(1990), 365-382.
	
\bibitem{Gourley2018}	S. A. Gourley, X. Lai, J. Shi, W. Wang, Y. Xiao and X. Zou, Role of white-tailed deer in
geographic spread of the blacklegged tick Ixodes scapularis: analysis of a spatially nonlocal model, Math. Biosci. Eng., 15 (2018), 1033-1054.
	
\bibitem{Hale1993} J. K. Hale and S. M. Verduyn Lunel,  {\it Introduction to Functional Differential Equations},  Springer, New York, 1993.
	
	
\bibitem{Hees} J. A. P.  Heesterbeek,  A brief history of $R_0$ and a recipe for	its calculation, {\it Acta Biotheoretica}, 50(2002), 189-204.
	
\bibitem{HSW} J. M. Heffernan, R. J.  Smith and L. M.  Wahl, Perspectives on the basic reproductive ratio, {\it J. R. Soc. Interface},  2(2005), 281-293.
	
	
\bibitem{HWZJDE} M. Huang, S.-L. Wu and X.-Q. Zhao,  The principal eigenvalue for partially degenerate  and 
	periodic reaction-diffusion systems with time delay,  {\it  J. Differential Equations}, in review (a revised version).
	
\bibitem{Inaba}  H. Inaba, On a new perspective of the basic reproduction number in heterogeneous environments, {\it J. Math. Biol.}, 65(2012), 309-348.
	
\bibitem{LZZJDDE19} X. Liang, L. Zhang and X.-Q. Zhao, Basic reproduction ratios for periodic abstract functional differential equations (with application to a spatial model for Lyme disease), 
	 {\it Journal of Dynamics and Differential Equations}, 31(2019), 1247-1278.
	
\bibitem{MS1990} R. H. Martin and H. L. Smith, Abstract functional-differential equations and reaction-diffusion systems, {\it Trans. Amer. Math. Soc.},  321(1990), 1-44.
	
\bibitem{MS1991}  R. H. Martin and H. L. Smith, Reaction-diffusion systems with time delays: monotonicity, invariance, comparison and convergence, {\it J. Reine Angew. Math.}, 413(1991), 1-35.   
	
\bibitem{SchJDE} K. Schumacher, On the resolvent of linear nonautonomous partial functional differential equations, 
{\it J. Differential Equations}, 59(1985), 355-387.
	
	
\bibitem{Smith95} H. L. Smith, {\it Monotone Dynamical Systems: An Introduction to the Theory of Competitive and Cooperative Systems},  Mathematical Surveys and
	Monographs 41, Amer. Math. Soc., Providence, RI, 1995.
	
\bibitem{ThiemeSIAP} H. R. Thieme,  Spectral bound and reproduction number for infinite-dimensional population structure and time heterogeneity,  {\it SIAM Journal of Applied Mathematics}, 70(2009), 188-211. 

\bibitem{ThZh} H. R. Thieme and X.-Q. Zhao, A  non-local  delayed  and  
diffusive  predator-prey  model, {\it Nonlinear Analysis, RWA}, 2(2001), 145-160.

\bibitem{van}  van den Driessche, Infectious Disease Modelling, 2017,
Reproduction numbers of infectious disease models, 	{\it Infectious Disease Modelling},  2 (2017), 288-303.

\bibitem{DriesscheWatmough2002} P. van den Driessche  and  J. Watmough,  Reproduction numbers and sub-threshold endemic equilibria for compartmental models of disease transmission,  {\it Mathematical Biosciences}, 180(2002), 29-48.
     
\bibitem{WangZhao1}  W. Wang and X.-Q. Zhao, Threshold dynamics for
     compartmental epidemic models in periodic environments, 
     {\it J. Dynamics and Differential Equations}, 20(2008), 699-717.
     
 \bibitem{WangZhao2} W. Wang and X.-Q. Zhao, Basic reproduction numbers for reaction-diffusion epidemic
     models, {\it SIAM J. Appl. Dyn. Syst.}, 11(2012), 1652-1673.
     
     
\bibitem{ZhaoJDDE} X.-Q. Zhao,  Basic reproduction ratios for periodic compartmental models with time delay, {\it Journal of Dynamics and Differential Equations}, 29(2017), 67-82.
     
 \bibitem{Zhao} X.-Q. Zhao, {\it Dynamical Systems in Population Biology}, second edition, Springer, New York, 2017.
     
 \bibitem{ZhaoJing} X.-Q. Zhao and  Z.-J. Jing,
   Global asymptotic  behavior of some cooperative systems
     of functional differential equations, {\it Canadian Appl. Math. Quart.}, 4(1996), 421-444.  
    
\end{thebibliography}
\end{document}